\documentclass[11pt]{article}
\usepackage{amsfonts,amssymb,amscd,amsmath,enumerate,verbatim,calc, times,url}
\usepackage{amscd,amsthm, psfrag,latexsym,epsfig,mdwlist,graphicx}
\usepackage[capitalise, noabbrev]{cleveref}
\usepackage{enumitem}
\usepackage{xcolor}
\usepackage{tikz}
\usetikzlibrary{positioning, shapes.geometric}

\usepackage{nomencl}
\usepackage{txfonts}
\usetikzlibrary{shadings}
\theoremstyle{plain}
\newtheorem{theorem}{Theorem}[section]
\newtheorem{lemma}[theorem]{Lemma}

\newtheorem{corollary}[theorem]{Corollary}
\theoremstyle{definition}
\newtheorem{definition}[theorem]{Definition}
\newtheorem{discussion}[theorem]{Discussion}

\newtheorem{example}[theorem]{Example}

\newtheorem{question}[theorem]{Question}

\newcommand{\N}{{\mathcal{N}}}                  
\newcommand{\GA}{\Gamma^{\vee}}                 
\newcommand{\lk}{{\rm{lk}}\ }                   
\newcommand{\rhk}[2]{\tilde{H}_{#1}(#2,k)}   
\newcommand{\xs}{x_1,\ldots,x_n}                

\newcommand{\m}{\mathbf{m}}                 

\newcommand{\be}{\beta}                          
\newcommand{\lcm}{{\mathop{\rm{lcm}}}}          
              
\newcommand{\st}{\ : \ }                        
\newcommand{\tuple}[1]{\langle #1 \rangle}      
\newcommand{\void}[1]{}

\newcommand{\erase}[1]{}

\newcommand{\LCM}{\mbox{LCM}}
 
\newcommand{\facets}{\mbox{Facets}} 
\newcommand{\NN}{\mathbb{N}} 

\newcommand{\PP}{\mathcal{P}}

\newcommand{\sm}{\setminus}

\renewcommand{\leq}{\leqslant}          
\renewcommand{\geq}{\geqslant}  

\date{\today}

\author{Sara Faridi\thanks{Department of
Mathematics and Statistics, Dalhousie University, Halifax, Canada, 
faridi@dal.ca} \hspace{1in}
Mayada Shahada \thanks{Department of
Mathematics, University of Bahrain, Sakheer, Kingdom of Bahrain, 
mshahada@uob.edu.bh}}

\title{\Large \sc Breaking up Simplicial Homology and 
  Subadditivity of Syzygies\thanks{Subject Classification 13A15, 13D02, 13F55, 05E40, 05E45} }

\begin{document}

\maketitle

\begin{abstract} We consider the following question: if a simplicial
  complex $\Gamma$ has $d$-homology, then
  does the corresponding $d$-cycle always induce cycles of smaller
  dimension that are not boundaries? We
    provide an answer to this question in a fixed
    dimension. We use the breaking of homology to show the
    subadditivity property for the maximal degrees of syzygies of
    monomial ideals in a fixed homological degree.
 \end{abstract}

\section{Introduction} 

The motivation for this paper is the subadditivity property for the
maximal degrees of syzygies of monomial ideals in polynomial rings.
Let $I$ be a homogeneous ideal in the polynomials ring
$S=k[x_1,\ldots,x_n]$ over a field $k$. Let $t_a$ denoted the
maximum value of $j$ such that the graded Betti number
$\beta_{a,j}(S/I) \neq 0$. The ideal $I$ satisfies the subadditivity
property on the maximal degrees of its syzygies if
\begin{equation}\label{e:subadditivity}
	t_{a+b} \leq t_a + t_b
\end{equation}
where $a+b$ is not more than the projective dimension of
the ideal.

The inequality in \eqref{e:subadditivity} arises most naturally in
the context of (Castelnuovo-Mumford) regularity, which, for 
the ideal $I$, can be described as the maximum value of $t_a-a$, for
all positive integers $a$. It has been shown to fail in general by
Avramov, Conca and Iyengar~\cite{ACI}, even if one restricts to
Cohen-Macaulay or even Gorenstein settings (see~\cite{JM2} for
examples and for a general survey on the topic). However, many
special cases are known: certain algebras with codimension $\leq 1$
(Eisenbud, Huneke and Ulrich~\cite{EHU}), certain classes of Koszul
rings (Avramov, Conca and Iyengar~\cite{ACI}), certain homological
degrees for Gorenstein algebras (El~Khoury and
Srinivasan~\cite{ES}), among others.

Avramov, Conca and Iyengar~\cite{ACI} conjectured that the subadditivity
property holds for Kozul rings and for all monomial ideals (it is
also open for toric ideals~\cite{JM2}). In the case of monomial
ideals, there are special cases for which~\eqref{e:subadditivity}
has been verified: when $a=1$ (Herzog and Srinivasan~\cite{HS}),
when $a=1,2,3$ and $I$ is generated in degree $2$ (Fern\'andez-Ramos
and Gimenez~\cite{FG}, Abedelfatah and Nevo~\cite{AN}),
Cohen-Macaulay ideals generated by monomials of degree $2$ when the
base field has characteristic $0$~\cite{ACI}, facet ideals of
simplicial forests (Faridi~\cite{F1}), ideals whose Betti diagram
has a special ``shape'' (Bigdeli and Herzog~\cite{BH}), several
classes of edge ideals of graphs and path ideals of rooted
trees~(Jayanthan and Kumar~\cite{JK}), and for $a$ where the
Stanley-Reisner complex of $I$ has dimension bounded by $t_a-a$
(Abedelfatah~\cite{A}).

In the case of monomial ideals, the syzygies can be characterized as
dimensions of homology modules of topological objects. This is one
of the central themes of Stanley-Reisner Theory, connecting
Commutative Algebra to Discrete Geometry and Topology. We refer the
reader to the books~\cite{BrH,P} for more details on these rich
connections.

By viewing the subadditivity property as a geometric one, the
inequality in~\eqref{e:subadditivity} can be shown to follow from the
following general type of question: \begin{center} \emph{Does a
		topological object with $d$-homology break into 
		sub-objects that have $a$-homology and $b$-homology, where $a$ and $b$ are related to $d$?}
\end{center}

This approach was taken by the first author in~\cite{F1}, where the
topological objects were atomic lattices ($\lcm$ lattices of monomial
ideals); see \cref{q:main} and \cref{q:main-general} below. In this
paper, using Hochster's formula (\cref{e:Hochster}), we examine this
problem from the point of view of the Stanley-Reisner complex, and we can provide a positive answer to the general
  question above for a fixed value of $d$. As a result we show that
subadditivity holds in a fixed homological degree for all monomial
ideals. The last section interprets the square-free results of the
paper for general monomial ideals.

\subsubsection*{Acknowledgements} The authors are grateful to the
referees and Kieran Bhaskara, whose comments improved the paper. The second author would
like to acknowledge the support of Dalhousie University where she was
a postdoc while this research was carried out. The research of the
first author was supported by an NSERC Discovery
Grant.

\section{Setup}

\subsection{The subadditivity property}

Throughout the paper, let $S=k[x_1,\ldots,x_n]$ be a polynomial ring
over a field $k$.  If $I$ is a graded ideal of $S$ with minimal free
resolution
$$ 0 \to \oplus_{j \in \NN} S(-j)^{\beta_{p,j}}\to \oplus_{j \in \NN}
S(-j)^{\beta_{p-1,j}} \to \cdots \to \oplus_{j \in \NN}
S(-j)^{\beta_{1,j}} \to S,$$ then for each $i$ and $j$, the rank
$\beta_{i,j}(S/I)$ of the free $S$-modules appearing above are called
the {\bf graded Betti numbers} of the $S$-module $S/I$.

If we let $$t_a=\max \{j \st \beta_{a,j}(S/I) \neq 0\},$$ a question
is whether the $t_a$ satisfy the {\bf subadditivity property}:
$t_{a+b} \leq t_a + t_b$?

The answer is known to be negative for a general homogeneous
ideal~\cite{ACI}, and unknown in the case of monomial ideals. For the
case of monomial ideals, there are special cases that are
known~\cite{HS, AN, FG, F1, BH, A}.

In the case of monomial ideals, Betti numbers can be interpreted as
the homology of objects in discrete topology: simplicial complexes,
order complexes of lattices, etc.; see for example~\cite{P} for a
survey of this approach. As a result, the subadditivity question can
be viewed as a question of breaking up homology in these objects. This
idea was explored in~\cite{F1} by the first author, where the
subadditivity problem was solved for facet ideals of simplicial
forests using homology of lattices.

By a method called \emph{polarization}~\cite{Fr} (see
  \cref{s:polarization} for the definition), one can reduce questions regarding
Betti numbers of monomial ideals to the class of {\bf square-free}
monomial ideals.

If $u \subset [n]=\{ 1,\ldots,n\}$, then  we define $$\m_u= \Pi_{i \in u} x_i$$
to be the  {\bf  square-free monomial with  support $u$}.

For our purposes it is useful to
consider a finer grading of the Betti numbers by indexing the Betti
numbers with monomials of the polynomial ring $S$.  A
\emph{multigraded} Betti number of $S/I$ is of the form
$\beta_{i,\m}(S/I)$ where $\m$ is a monomial in $S$
and $$\beta_{i,j}(S/I)= \sum_{u \subseteq [n] \mbox{ and }|u|=j}
\beta_{i,\m_u}(S/I).$$

\subsection{Simplicial Complexes}
A {\bf simplicial complex} $\Gamma$ on a set $W$ is a set of subsets
of $W$ with the property that if $F\in \Gamma$ then for every subset
$G \subseteq F$ we have $G \in \Gamma$. Every element of $\Gamma$
is called a {\bf face}, the maximal faces under inclusion are called
{\bf facets}, and a simplicial complex contained in $\Gamma$ is called
a {\bf subcomplex} of $\Gamma$. The set of all $v \in W$ such that $\{v\}
\in \Gamma$ is called the {\bf vertex set} of $\Gamma$, and is denoted
by $V(\Gamma)$. The set of facets of $\Gamma$ is denoted by
$\facets(\Gamma)$. If $\facets(\Gamma)=\{F_1,\ldots,F_q\}$, then we
denote $\Gamma$ by $$\Gamma=\tuple{F_1,\ldots,F_q}.$$ If $A \subset
V(\Gamma)$, then the {\bf induced subcomplex} $\Gamma_A$ is defined
as $$\Gamma_A=\{F \in \Gamma \st F \subseteq A\}.$$ 
The {\bf Alexander dual} $\GA$ of $\Gamma$, if we set $F^c=V(\Gamma)
\sm F$, is defined as $$\GA
=\{F \subseteq V (\Gamma) \st F^c \notin \Gamma\}
=\{V (\Gamma) \sm F   \st F \notin \Gamma\}.
$$  The {\bf link} of a face $F$ of $\Gamma$
is $$\lk_\Gamma(F)=\{G \in \Gamma \st F \cap G = \emptyset \mbox{ and
} F \cup G \in \Gamma \}.$$

If $I$ is a square-free monomial ideal in $S$, it corresponds uniquely to a
simplicial complex $$\N(I)=\{u \subset [n] \st m_u \notin I\}$$
called the {\bf Stanley-Reisner complex} of $I$.  Conversely, if $\Gamma$
is a simplicial complex whose vertices are labelled with $\xs$, then
one can associate to it its unique {\bf Stanley-Reisner
  ideal} $$\N(\Gamma)=\{m_u \st u \subset [n] \mbox{ and } u \notin
\Gamma\}.$$

The uniqueness of the Stanley-Reisner correspondence implies that
$$\N(\Gamma)=I \iff \N(I)=\Gamma.$$

\subsection{The lcm lattice}

A \emph{lattice} is a partially ordered set where
every two elements have a greatest lower bound called their
\emph{meet} and a lowest upper bound called their \emph{join}. A
bounded lattice has an upper and a lower bound denoted by $\hat{1}$
and $\hat{0}$, respectively.

If $L$ is a lattice with $r$ elements, then the {\bf
    order complex} of $L$ is the simplicial complex on $r$
  vertices, where the elements of each  chain in $L$ form a face.

If $I$ is a monomial ideal, then the {\bf lcm lattice} of $I$, denoted
by $\LCM(I)$, is a bounded lattice ordered by
  divisibility, whose elements are the generators of
  $I$ and their least common multiples, and the meet of two elements
is their least common multiple.

Two elements of a lattice are called {\bf complements} if their join
is $\hat{1}$ and their meet is $\hat{0}$. If the lattice is $\LCM(I)$,
then it was shown in~\cite{F1} that two monomials
in $\LCM(I)$ are complements if their gcd is not in $I$ and their lcm
is the lcm of all the generators of $I$.

Gasharov, Peeva and Welker~\cite{GPW,P} showed that multigraded Betti numbers of
$S/I$ can be calculated from the homology of (the order complex of)
 the lattice $\LCM(I)$: if $\m$ is a monomial in $L=\LCM(I)$, then 
$$ \beta_{i,\m}(S/I)=\dim_k \tilde{H}_{i-2} \left ( (1, \m)_{L}; k \right ) $$
where $(1, \m)_{L}$ refers to the subcomplex of the
  order complex consisting of all nontrivial monomials in $L$
strictly dividing $\m$.

On the other hand, in a 1977 paper, Baclawski~\cite{B} showed that if $L$ is a
finite lattice whose proper part has nonzero homology, then every
element of $L$ has a complement.

The following question was raised in~\cite{F1} as a
  potential way to answer the subadditivity question.

\begin{question}\label{q:main} If $I$ is a square-free monomial ideal
  in variables $\xs$, and $\beta_{i,n} (S/I)\neq 0$, $a,b >0$ and
  $i=a+b$, are there complements $\m$ and $\m'$ in $\LCM(I)$ with
  $\beta_{a,\m}(S/I) \neq 0$ and $\beta_{b,\m'}(S/I) \neq 0$?
\end{question}

Considering  that it is enough to study the ``top degree''
 Betti numbers (those of degree $n$, in this case)~\cite{EF1,F1},
a positive answer to \cref{q:main} will establish the
subadditivity property for all monomial ideals, since
$$t_a+t_b \geq \deg(\m)+ \deg(\m') \geq n=t_i.$$

\cref{q:main} can be written more generally as a
  question about the homology of the $\lcm$ lattice, or
  in fact, any finite lattice.

  \begin{question}\label{q:main-general} If $L=\LCM(I)$ and 
$\tilde{H}_{i-2} \left ( (1,
      x_1 \cdots x_n)_{L}; k \right ) \neq 0$, $a,b>0$ and $i=a+b$,
    are there complements $\m$ and $\m'$ in $\LCM(I)$ with
    $\tilde{H}_{a-2} \left ( (1,\m)_{L}; k \right ) \neq 0$ and
    $\tilde{H}_{b-2} \left ( (1,\m')_{L}; k \right ) \neq
    0$?  \end{question}

With the same idea, one could translate \cref{q:main} into breaking up
simplicial homology using Hochster's formula.

\subsubsection{Hochster's Formula}

Let $I=(\m_1,\ldots, \m_q)$ be a square-free monomial ideal 
  in the polynomial ring $S=k[\xs]$.
Hochster's formula (see for example~\cite[Cor.~8.1.4 and Prop.~5.1.8]{HH}) states that if
$I=\N(\Gamma)$ and $\m_u$ a monomial, then
\begin{equation}\label{e:Hochster}
  \beta_{i,\m_u}(S/I)=\dim_k
\rhk{i-2}{\lk_{\GA}(u^c)}=\dim_k\rhk{|u|-i-1}{\Gamma_u}
\end{equation}
where $u^c=[n]\sm u$ is the set complement of $u$. We would now like to reinterpret~\cref{q:main} in the language of
Hochster's formula. To begin with, since we are dealing with
square-free monomials, we can consider a monomial $\m_u$ equivalent
to the set $u$ and use  intersections for $\gcd$,
unions for $\lcm$, and $\m_u^c$ for $u^c$.

Suppose
 $$\beta_{i,x_1\ldots
   x_n}(S/I)=\dim_k \rhk{i-2}{\lk_{\GA}(\emptyset)}=\dim_k
 \rhk{i-2}{\GA} \neq 0$$ and $i=a+b$ where $a,b >0$.  We would like to know if there are complements $\m, \m' \in \LCM(I)$ such
 that $$\beta_{a, \m}(S/I)\neq 0 \mbox{ and } \beta_{b, \m'}(S/I)\neq
 0.$$

 First observe that, $\GA=\tuple{\m_1^c,\ldots,\m_q^c}$
   (e.g.~\cite{HH} or \cite[Prop. 2.4]{F2}).

 We have
\begin{align*} 
  \m \in \LCM(I) \iff & \m=\m_{i_1}\cup \m_{i_2}\cup \cdots \cup \m_{i_s} \mbox{ for some  } 1 \leq i_1 < i_2 < \cdots < i_s \leq q \\
  \iff & \m^c = \m_{i_1}^c\cap \m_{i_2}^c\cap \cdots \cap \m_{i_s}^c \mbox{ for some  }
  1 \leq i_1 < i_2 < \cdots < i_s \leq q \\
  \iff & \m^c \mbox{ is the intersection of some facets of } \GA.
\end{align*}

Moreover, if $\m, \m' \in \LCM(I)$, then 
\begin{align*}
  \m, \m' \mbox{are complements} \iff & \m \cup \m' =[n]
  \mbox{ and }  \m \cap \m' \notin I\\
  \iff & \m^c \cap \m'^c = \emptyset \mbox{ and } \m \cap \m' \in \Gamma\\
  \iff & \m^c \cap \m'^c = \emptyset \mbox{ and } (\m \cap \m')^c \notin \GA\\
  \iff & \m^c \cap \m'^c = \emptyset \mbox{ and } \m^c \cup \m'^c \notin \GA.\\
\end{align*}

So we are looking for subsets $A,B \subseteq [q]$ such that
\begin{enumerate}
\item $\m^c=\bigcap_{j \in A} \m_j^c$ and $\m'^c=\bigcap_{j \in B} \m_j^c $
\item $\m^c \cap \m'^c = \emptyset$
\item $\m^c \cup \m'^c \notin \GA$
\item $\rhk{a-2}{\lk_{\GA}(\m^c)} \neq 0$
  and $\rhk{b-2}{\lk_{\GA}(\m'^c)} \neq 0$.
\end{enumerate}

Now we can state \cref{q:main} in the following form.

\begin{question}\label{q:main-H-link} If $\Gamma=\tuple{F_1,\ldots,F_q}$ is a 
  simplicial complex with $\rhk{i-2}{\Gamma} \neq 0$ and $i=a+b$ where
  $a,b >0$, can we find subsets $A,B \subseteq [q]$ such
  that
\begin{enumerate}
\item $F=\bigcap_{j \in A} F_j$ and $G=\bigcap_{j \in B} F_j $
\item $F \cap G = \emptyset$ 
\item $F \cup G \notin \Gamma$
\item $\rhk{a-2}{\lk_{\Gamma}(F)} \neq 0$ and $\rhk{b-2}{\lk_{\Gamma}(G)} \neq 0$?
\end{enumerate}
\end{question}

  \begin{example} \label{e:run-example}If $\N(I)^{\vee}=\Gamma=\langle
    xzu,xzv,xuv,yzu,yzv,yuv,xy \rangle$,

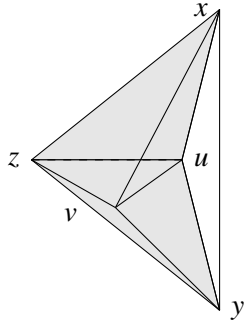
\[ 
\begin{tikzpicture}
\tikzstyle{point}=[inner sep=0pt]
\node (x)[point,label=left:$x$] at (1,2,0) {};
\node (y)[point,label=right:$y$] at (1,-2,0) {};
\node (z)[point,label=left:$z$] at (-1.5,0,0) {};
\node (u)[point,label=right:$u$] at (.5,0,0) {};
\node (v)[point] at (0,-.25,1) {};
\draw [fill=gray!20](x.center) -- (v.center) -- (z.center);
\draw [fill=gray!20](x.center) -- (v.center) -- (u.center);
\draw [fill=gray!20](x.center) -- (u.center) -- (z.center);
\draw [fill=gray!20](y.center) -- (v.center) -- (z.center);
\draw [fill=gray!20](y.center) -- (v.center) -- (u.center);
\draw [fill=gray!20](y.center) -- (u.center) -- (z.center);
\draw (x.center) -- (y.center);
\draw (x.center) -- (z.center);
\draw (x.center) -- (v.center);
\draw (x.center) -- (u.center);
\draw (y.center) -- (z.center);
\draw (y.center) -- (u.center);
\draw (y.center) -- (v.center);
\draw (z.center) -- (v.center);
\draw[dashed] (z.center) -- (u.center);
\draw (u.center) -- (v.center);
\node (f)[point,label=left:$v$] at (-.35,-.35,1) {};
\end{tikzpicture} \]

then  $I=(xz,yz,xu,yu,xv,yv,zuv)$ has  Betti table
\bigskip

$$\begin{array}{rlllll}
&0 &1 &2 &3 &4\\
\mbox{total}:&1 &7 &11 &6 &1\\
0:&1 &. &. &. &.\\
1:&. &6 &9 &5 &1 \\
2:&. &1 &2 &1 &.
\end{array}$$

\bigskip

So $\beta_{i,xyzuv} \neq 0$ when $i=3,4$, which corresponds to
nonvanishing of homology of links of faces of $\Gamma$ in dimensions
$1,2$. We consider each case separately:

    \begin{enumerate}

    \item $i=3$, $a=1$, $b=2$.  Then $\rhk{1}{\Gamma}\neq0$. Let
      $F=xy$ and $G=xuv \cap yuv =uv$, then $F\cap G = \emptyset$, $F \cup G=xyuv \notin \Gamma$, and $$\rhk{a-2}{\lk_{\Gamma}(F)} =\rhk{-1}{\tuple{\emptyset}} \neq 0 \mbox{ and } \rhk{b-2}{\lk_{\Gamma}(G)}=\rhk{0}{\tuple{x,y}} \neq 0.$$

  \item $i=4$, $a=1$, $b=3$. Then $\rhk{2}{\Gamma}\neq0$. Let $F=yzu$
    and $G=xzu \cap xuv \cap xzv  \cap xy =x$, then $F\cap G = \emptyset$, $F \cup
    G=xyzu \notin \Gamma$, and $$\rhk{a-2}{\lk_{\Gamma}(F)}
    =\rhk{-1}{\tuple{\emptyset}} \neq 0 \mbox{ and }
    \rhk{b-2}{\lk_{\Gamma}(G)}=\rhk{1}{\tuple{zu,uv,zv,y}} \neq 0.$$

  \item $i=4$, $a=2$, $b=2$. Then $\rhk{2}{\Gamma}\neq0$. Let $F=yzu \cap yuv=yu$
    and $G=xzu \cap xzv =xz$, then $F\cap G = \emptyset$, $F \cup
    G=xyzu \notin \Gamma$, and $$\rhk{a-2}{\lk_{\Gamma}(F)}
    =\rhk{0}{\tuple{z,v}} \neq 0 \mbox{ and }
    \rhk{b-2}{\lk_{\Gamma}(G)}=\rhk{0}{\tuple{u,v}} \neq 0.$$
    \end{enumerate}
\end{example}

A dual version of \cref{q:main-H-link} can be stated as
  follows (see \cref{c:dual-version} for the justification).

\begin{question}\label{q:main-H} If  $\Gamma$ is a simplicial
  complex on the vertex set $\{\xs\}$, and $\rhk{i-2}{\Gamma} \neq 0$,
  and $n-i+1=a+b$, where $a$ and $b$ are positive integers, are there
  nonempty subsets $C, D \subseteq \{\xs\}$ such that

\begin{enumerate}
\item $C\cup D  =\{\xs\}$
\item $C\cap D \in \Gamma$
\item $\rhk{|C|-a-1}{\Gamma_C} \neq 0$   and  $\rhk{|D|-b-1}{\Gamma_D} \neq 0$?
\end{enumerate}
\end{question}

\begin{example}
Let $\N(I)=\Gamma=\langle zwx,vwx,uvx,zux,zuy,uvy,vwy,zwy\rangle$.

\[ \begin{tabular}{ccccc}
\begin{tikzpicture}
\tikzstyle{point}=[inner sep=0pt]
\node (z)[point,label=left:$z$] at (0.1,0) {};
\node (u)[point,label=right:$u$] at (2,0) {};
\node (v)[point,label=right:$v$] at (2.8,0.7) {};
\node (w)[point,label=left:$w$] at (1,0.7) {};
\node (x)[point,label=above:$x$] at (1.5,2.1) {};
\node (y)[point,label=below:$y$] at (1.7,-1.3) {};
\node (uu)[point,label=right:$u$] at (2.3,-0.1) {};
\node (ww)[point,label=left:$w$] at (0.7,0.9) {};
\draw [fill=gray!20](z.center) -- (x.center) -- (u.center);
\draw [fill=gray!20](u.center) -- (v.center) -- (x.center);
\draw (z.center) -- (x.center) -- (u.center);
\draw[dashed] (z.center) -- (w.center) -- (v.center);
\draw[dashed] (w.center) -- (x.center);
\draw [fill=gray!20](z.center) -- (u.center) -- (y.center) -- (z.center);
\draw [fill=gray!20](v.center) -- (y.center) -- (u.center);
\draw[dashed] (w.center) -- (y.center);
\draw (uu.center);
\draw (ww.center);
\end{tikzpicture} &\hspace{0.7 in} &\begin{tikzpicture}
\tikzstyle{point}=[circle,thick,draw=black,fill=black,inner sep=0pt]
\node (x)[point,label=above:$x$] at (1.5,1) {};
\node (y)[point,label=below:$y$] at (1.5,-1) {};
\draw (x.center);
\draw (y.center);
\end{tikzpicture}  &\hspace{0.7 in} &\begin{tikzpicture}
\tikzstyle{point}=[circle,thick,draw=black,fill=black,inner sep=0pt]
\node (z)[point,label=below:$z$] at (1.1,0) {};
\node (u)[point,label=below:$u$] at (3,0) {};
\node (v)[point,label=above:$v$] at (3.8,0.7) {};
\node (w)[point,label=above:$w$] at (2,0.7) {};
\draw (z.center) -- (u.center) -- (v.center) -- (w.center) -- cycle;
\end{tikzpicture}\\
&&\\
$\Gamma$ && $\Gamma_C$ && $\Gamma_D$
\end{tabular}\]

Then $I=(xy,zv,uw)$ has Betti table

$$\begin{array}{rllll}
&0 &1 &2 &3 \\
\mbox{total}:&1 &3 &3 &1\\
0:&1 &. &. &. \\
1:&. &3 &. &. \\
2:&. &. &3 &.\\
3:&. &. &. &1
\end{array}$$

\bigskip

So $\be_{3,xyzuvw}(S/I)\neq 0$ which corresponds to nonvanishing homology of $\Gamma$ in dimension $2$ (i.e. $\rhk{2}{\Gamma}\neq 0$). Let $a=1$ and $b=2$. Choose $C=\{x,y\}$ and $D=\{z,u,v,w\}$. Then $C\cup D=\{x,y,z,u,v,w\}$, $C\cap D=\emptyset \in \Gamma$ and

\begin{center}
	$\rhk{|C|-a-1}{\Gamma_C}=\rhk{0}{\langle x,y \rangle} \neq 0$ and $\rhk{|D|-b-1}{\Gamma_D}=\rhk{1}{\langle zu,uv,vw,zw \rangle}\neq 0$.
	\end{center}	
	
\end{example}
A positive answer to either \cref{q:main-H-link} or \cref{q:main-H}
would settle the subadditivity question for syzygies.

\section{Main results}

The following lemma is an easy exercise.

\begin{lemma}\label{l:link-lemma} $\Gamma$ simplicial complex and
  $A \in \Gamma$ and $B\in\lk_\Gamma(A)$, then
  $$\lk_{\lk_\Gamma(A)}(B)=\lk_\Gamma(A\cup B).$$
\end{lemma}

In a simplicial complex $\Gamma$ we say a $d$-cycle
$\Sigma$ is {\bf supported on} faces $F_1,\ldots,F_q$ if
$\Sigma=a_1F_1+ \cdots + a_qF_q$ for nonzero scalars $a_1,\ldots,a_q
\in k$. We say that $\Sigma$ is   {\bf a face-minimal cycle} or {\bf minimally supported on} $F_1,\ldots,F_q$ if additionally no proper subset of $F_1,\ldots,F_q$ is the support of a $d$-cycle. If $\Sigma$ is supported on
$F_1\ldots,F_q$, we call the simplicial complex $\tuple{F_1,\ldots,F_q}$ the {\bf support complex of $\Sigma$}.

\cref{e:guide1} can guide the reader through the statement of the
theorem below, a variation of which appears as Theorem 4.2
of~\cite{RW}.

\begin{theorem}\label{t:new-link-going-down} Let $k$ be a field, 
  $\Gamma$ a $d$-dimensional simplicial complex, and
  $$\Sigma=a_1F_1 + \cdots + a_qF_q \hspace{.5in} a_1, \ldots, a_q \in
  k$$ a $d$-cycle in $\Gamma$ supported on
  $F_1,\ldots,F_q$ which is not a boundry, so that $\rhk{d}{\Gamma} \neq 0$. Suppose $A$ is a face of the support complex of $\Sigma$ such
  that for some $s \leq q$ we have $$A \subseteq F_1\cap \ldots \cap
  F_s, \mbox{ and } A \not \subseteq F_j \mbox{ if } j>s$$ and $0 \leq |A| \leq d+1$. Then
 \begin{enumerate}
 \item there are $\epsilon_i \in \{\pm 1\}$ for
  $i=1,\ldots,s$ such that
  $$\Sigma_A=\epsilon_1a_1(F_1\sm A)+\cdots + \epsilon_sa_s(F_s\sm A)$$ 
  is a $(d-|A|)$-cycle in $\lk_\Gamma(A)$ that is not a boundary in
  $\lk_\Gamma(A)$; 
 \item $\rhk{d-|A|}{\lk_\Gamma(A)}\neq 0$;
 \item $A = F_1\cap \ldots \cap F_s$.
 \end{enumerate}
 \end{theorem}

 \begin{proof} The case $|A|=d+1$ will result in $\lk_\Gamma(A)=\{\emptyset\}$ which has $(-1)$-homology. So we can assume that $0 \leq |A| \leq d$. To prove Statement~1 we will proceed using induction on
   $a=|A|$. If $a=0$, then $\lk_\Gamma(A)=\Gamma$, $\Sigma_A=\Sigma$
   and there is nothing to prove.

   Suppose $a>0$, $A=\{v_1,\ldots, v_a\}$, $A'=\{v_1,\ldots,
   v_{a-1}\}$ (or $A'=\emptyset$ when $a=1$) and
   $\Gamma'=\lk_\Gamma(A')$, and suppose without loss of
   generality $$A' \subseteq F_1\cap \ldots \cap F_t \mbox{ and } A'
   \not \subseteq F_j \mbox{ for } j>t\geq s.$$ By the induction
   hypothesis, for some $\epsilon_i'\in \{\pm 1\}$ there is a
   $(d-(a-1))$-cycle
$$\Sigma_{A'}=a_1\epsilon_1'(F_1\sm A')+\cdots + a_t\epsilon_t'(F_t\sm
A')$$ in $\Gamma'$ that is not a boundary in $\Gamma'$ and
$\rhk{d-(a-1)}{\Gamma'}\neq 0$. In particular, we
  must have $t\neq s$ as otherwise the support complex of $\Sigma_{A'}$
  would be a cone with every facet containing $v_a$, a contradiction.

  We know that $v_a \in (F_i\sm A')$ if and only if $i\leq
  s$. Depending on the orientation of the faces of the complex
  $\Gamma'$, for some $\epsilon_{i}''\in \{\pm 1\}$, we can write
   $$\begin{array}{rl}
   0=&\partial(\Sigma_{A'})\\
   &\\
    =&\epsilon_1'a_1\partial(F_1\sm A')+ \cdots + \epsilon_t'a_t\partial(F_t \sm A')\\
   &\\
    =&\epsilon_1''\epsilon_1'a_1(F_1 \setminus A)+ \cdots +
    \epsilon_s''\epsilon_s'a_s(F_s \setminus A) + \mathcal{U} +
    \partial(\epsilon_{s+1}'a_{s+1}F_{s+1} \sm A'+ \cdots + \epsilon_t'a_tF_t \sm A')
  \end{array}$$
  where $\mathcal{U}$ consists of all the summands above which contain 
  the vertex $v_a$, and hence 
    $$\mathcal{U}=\sum_{j=1}^s \epsilon_j'a_j\left(\partial(F_j \sm
  A') - \epsilon_j''F_j \sm A \right )=0.$$

    If we set $\epsilon_i=\epsilon_i''\epsilon_i'$ and
    $\Sigma_A=\epsilon_1a_1(F_1 \sm A) + \cdots + \epsilon_sa_s(F_s
    \sm A)$ it follows that
      $$\Sigma_A= -\partial(\epsilon_{s+1}'a_{s+1}(F_{s+1} \sm A')+ \cdots + 
    \epsilon_t'a_t(F_t \sm A'))$$ and
    $$\partial(\Sigma_A)=- \partial^2(\epsilon_{s+1}'a_{s+1}(F_{s+1}
    \sm A')+ \cdots + \epsilon_t'a_t(F_t \sm A'))=0.$$ So $\Sigma_A$
    is a $(d-a)$-cycle in $\lk_{\Gamma'}(v_a)=\lk_\Gamma(A)$ by
    \cref{l:link-lemma} (and since $v_a \in \Gamma'$). Since
    $\dim(\lk_\Gamma(A))=d-|A|$, the $(d-|A|)$-cycle $\Sigma_A$ is not
    a boundary in $\lk_\Gamma(A)$. Therefore, $\rhk{d-|A|}{\Gamma}\neq
    0$, proving Statement 2.

      To see Statement~3, note that if $F_1\ldots,F_s$ all contain a vertex
      outside $A$, then the support complex of $\Sigma_A$ would be a cone, contradicting Statement~2.
 \end{proof}

\begin{example}\label{e:guide1}   Let $\Gamma=\langle
    xy,zu,zv,uv \rangle$, which is the Alexander dual of the simplicial complex $\Gamma$ in \cref{e:run-example}.

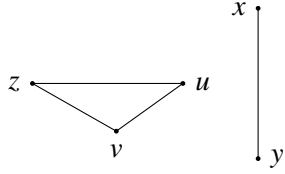
\begin{center}
	\begin{tikzpicture}
\tikzstyle{point}=[circle,thick,draw=black,fill=black,inner sep=0pt]
		\node (x)[point,label=left:$x$] at (1.5,1,0) {};
		\node (y)[point,label=right:$y$] at (1.5,-1,0) {};
		\node (z)[point,label=left:$z$] at (-1.5,0,0) {};
		\node (u)[point,label=right:$u$] at (.5,0,0) {};
		\node (v)[point,label=below:$v$] at (0,-.25,1) {};
		\draw (x.center) -- (y.center);
		\draw (z.center) -- (v.center);
		\draw (z.center) -- (u.center);
		\draw (u.center) -- (v.center);
	\end{tikzpicture}
\end{center}

As stated in \cref{t:new-link-going-down}, $\Gamma$ is a $1$-dimensional simplicial complex and has $\Sigma=uz+zv+vu$ as a $1$-cycle so that $\rhk{1}{\Gamma} \neq 0$. Taking $A=\{z\}$, then $\Sigma_A=u-v$ is a $0$-cycle in $\lk_{\Gamma}(A)=\tuple{u,v}$ with $\rhk{0}{\lk_{\Gamma}(A)} \neq 0$. 


\end{example}


\begin{corollary}\label{c:intersection}
  Let $k$ be a field, $\Gamma$ a $d$-dimensional simplicial complex
  with $\rhk{d}{\Gamma}\neq 0$, and let $\Sigma$ be a 
  $d$-cycle in $\Gamma$ which is not a boundary. Let $A$ be a face of the support complex of $\Sigma$, and suppose $F_1\ldots,F_q$ are the facets of $\Gamma$
  that contain $A$. Then $$A=\displaystyle \bigcap_{j=1}^q F_j.$$
  \end{corollary}

  \begin{proof} Since
    $\lk_\Gamma(A)=\tuple{F_1 \sm A, \ldots, F_q \sm A}$, if there is
    a vertex of $\displaystyle \bigcap_{j=1}^q F_j$ which is not in
    $A$, then $\lk_\Gamma(A)$ would be a cone, and would therefore
    have no homology, contradicting \cref{t:new-link-going-down}.
  \end{proof}

\cref{t:subadd-link} below is a formal statement on breaking
homological cycles. We refer the reader to parts (2) and (3) of
\cref{e:run-example} where we demostrated the theorem's
statement. Note also that the case in part (1) of \cref{e:run-example}
follows the same pattern, though a proof is not known yet.

\begin{theorem}[{\bf Breaking up cycles on links}]\label{t:subadd-link}
  Let $k$ be a field and $\Gamma=\tuple{F_1,\ldots,F_r}$ be a
  $d$-dimensional simplicial complex such that $$\rhk{d}{\Gamma} \neq 0 \mbox{ and } d+2=a+b \mbox{ for some }
    a,b>0.$$ Suppose $\Gamma$ contains a $d$-dimensional
    cycle $$\Sigma=\displaystyle\sum_{j=1}^qa_jF_j$$ supported on the
    facets $F_1,\ldots,F_q$ of $\Gamma$, and $\Sigma$ is not boundary in
    $\Gamma$. Then there are subsets $A,B \subseteq [q] \subseteq [r]$ with $$F=\bigcap_{j \in A} F_j \mbox{ and }
  G=\bigcap_{j \in B} F_j$$ such that
\begin{enumerate}
\item $F \cap G = \emptyset$;
\item $F \cup G \notin \Gamma$; 
\item $\rhk{a-2}{\lk_{\Gamma}(F)} \neq 0$ and $\rhk{b-2}{\lk_{\Gamma}(G)} \neq
  0$.  \suspend{enumerate} Moreover, if $a,b>1$, $F$ and $G$ and $\epsilon_j, \delta_j \in \{\pm 1\}$ could be chosen to additionally satisfy: 
  \resume{enumerate}
\item  $|F|=b$ and $|G|=a$;

 \item $\Sigma_F=\displaystyle\sum_{j \in A} \epsilon_ja_j \left (F_j \sm F\right )$ is an $(a-2)$-cycle in $\lk_{\Gamma}(F)$ which is
  not a boundary ;
\item   $\Sigma_G=\displaystyle\sum_{j \in B} \delta_j a_j \left (F_j \sm G \right )$ is a
  $(b-2)$-cycle in $\lk_{\Gamma}(G)$ which is not a boundary.
\end{enumerate}
\end{theorem}

\begin{proof} Set $i=d+2$. We first consider the case $b=1$ and $a=i-1$. If $a=1$, then $d=0$ and $\Gamma$ is disconnected. Let $F$ and $G$ be two facets each belonging to a distinct connected component of
    $\Gamma$. Then we clearly have $F\cap G=\emptyset$ and $F \cup G
    \notin \Gamma$. Moreover, $\lk_{\Gamma}(F)=\lk_{\Gamma}(G)=\{\emptyset\}$ and so 
    $$\rhk{a-2}{\lk_{\Gamma}(F)}=\rhk{b-2}{\lk_{\Gamma}(G)}= \rhk{-1}{\{ 
      \emptyset \}}\neq 0$$ as desired.

    If $b=1$ and $a=i-1>1$, then $d=a+b-2>0$. By
  \cref{t:new-link-going-down}, if we take a vertex $v$ in the support
  complex of $\Sigma$, then $\rhk{i-3}{\lk_\Gamma(v)} \neq
  0$.

  Since $\Sigma$ is a cycle, not all of $F_1, \ldots,F_q$ contain $v$.
  Let $G$ be one of the facets $F_1,\ldots, F_q$ that does not
  contain $v$. Then $F\cap G =\emptyset$ and $F\cup G \notin \Gamma$ (as
  $G$ is a facet), and
  moreover $$\rhk{a-2}{\lk_\Gamma(F)}=\rhk{i-3}{\lk_\Gamma(v)}
  \neq 0 \mbox{ and } \rhk{b-2}{\lk_\Gamma(G)}=
  \rhk{-1}\{\emptyset\} \neq 0.$$
   
     Now suppose $a, b \geq 2$ and $a=i-b$. Suppose
     $F_1=\{w_1,v_1,\ldots,v_{i-2}\}$. Then since $F_1$ is in the
     support of the $(i-2)$-cycle $\Sigma$, $\{w_1,
     v_2,\ldots,v_{i-2}\}$ must appear in another one of the $F_j$ in
     the support of $\Sigma$, say $F_2$. Suppose
     $F_2=\{w_1,w_2,v_2,\ldots,v_{i-2}\}$. Considering that $a=i-b
     \leq i-2$, let $$G=\{v_1,\ldots,v_a\}
     \mbox{ and } F=\{v_{a+1},\ldots,v_{i-2},w_1,w_2\}.$$

     Then $|G|=a$ and $|F|=i-2+2-a=b$. Moreover $F \cap
       G =\emptyset$ by construction, and if $i-2=d$, then $F\cup G
       \notin \Gamma$ since $|F \cup G|=d+2$ which is larger than the size
       of any face of $\Gamma$.

    By \cref{t:new-link-going-down}, and noting that
    $i-2-|G|=b-2$ and $i-2-|F|=a-2$, we have
    $$\rhk{a-2}{\lk_\Gamma(F)}\neq 0 \mbox{ and } \rhk{b-2}{\lk_\Gamma(G)}\neq 0,$$ conditions~5 and~6 are satisfied, and if $$A=\{ j\in [q] \st F
    \subset F_j\} \mbox{ and } B=\{j \in [q] \st G \subset F_j\}$$
    then $$F=\bigcap_{j \in A} F_j \mbox{ and } G=\bigcap_{j \in B}
    F_j.$$
    
\end{proof}

Another version of \cref{t:subadd-link} below is one which gives
lower-dimensional cycles in induced subcomplexes.

\begin{corollary}[{\bf Breaking up cycles}]\label{c:dual-version} Let $\Gamma$ be a simplicial
  complex on the vertex set $\{\xs\}$, and suppose
  $\rhk{d-2}{\Gamma} \neq 0$, where $d$ is
    the smallest possible size of a nonface of $\Gamma$. Suppose
    $n-d+1=a+b$, where $a$ and $b$ are positive integers.  Then there
  are nonempty subsets $C, D \subseteq \{\xs\}$ such that

\begin{enumerate}
\item $C\cup D  =\{\xs\}$;
\item $C\cap D \in \Gamma$;
\item $\rhk{|C|-a-1}{\Gamma_C} \neq 0$   and  $\rhk{|D|-b-1}{\Gamma_D} \neq 0$.
\end{enumerate}
\end{corollary}

\begin{proof} By Alexander duality - see Prop.~5.1.10
    and the discussion preceding Prop.~5.1.8 in~\cite{HH}- we have
  that $\rhk{n-d-1}{\GA} \neq 0$. Now $d$ is the smallest
    possible size of a nonface of $\Gamma$, so by the definition of Alexander duals, $\dim(\GA)=n-d-1$.

  Suppose $\GA=\tuple{F_1,\ldots,F_r}$. 
  If $n-d+1=a+b$, then, by \cref{t:subadd-link}, there are subsets $A$ and $B$ of $[r]$ 
   such that $$F=\bigcap_{j \in A} F_j \mbox{ and }
  G=\bigcap_{j \in B} F_j$$ and 
\begin{enumerate}
\item[(i)] $F \cap G = \emptyset$;
\item[(ii)] $F \cup G \notin \GA$; 
\item[(iii)] $\rhk{a-2}{\lk_{\GA}(F)} \neq 0$ and $\rhk{b-2}{\lk_{\GA}(G)} \neq
  0$.  \end{enumerate}

Now let $$C=F^c=\bigcup_{j \in A} F_j^c \mbox{ and } D=G^c=\bigcup_{j
  \in B} F_j^c.$$ Then by~(i), $C\cup D=(F \cap G) ^c=\{\xs\}$.
By~(ii), $(C\cap D)^c= F\cup G \notin \GA$ so $C\cap D \in
\Gamma$. Finally by~(iii) and \cref{e:Hochster},
$\rhk{|C|-a-1}{\Gamma_C} \neq 0$ and $\rhk{|D|-b-1}{\Gamma_D} \neq 0$.
\end{proof}

\begin{theorem}[{\bf Subadditivity of syzygies of square-free monomial ideals}]\label{t:subadditivity} If $I$ is a square-free monomial ideal in the polynomial ring $S=k[\xs]$ where
  $k$ is a field, and $d$ is the smallest possible degree of a
  generator of $I$. Suppose $i=n-d+1$, 
  $\beta_{i,n}(S/I) \neq 0$ and $i=a+b$, for some positive integers
$a$ and $b$. Then $t_i \leq t_a + t_b$.
        \end{theorem}

        \begin{proof} By Hochster's
          formula (\cref{e:Hochster}), if $\Gamma=\N(I)$,
          then $$\be_{n-d+1,n}(S/I)=\beta_{n-d+1,x_1\cdots x_n}(S/I)=\dim_k
          \rhk{d-2}{\Gamma}\neq 0.$$

          If $n-d+1=a+b$, then by \cref{c:dual-version}, there are
          nonempty subsets $C, D \subseteq \{\xs\}$ such that
$$C\cup D  =\{\xs\} \mbox{ and } C\cap D \in \Gamma,$$
  and 
$$\rhk{|C|-a-1}{\Gamma_C} \neq 0 \mbox{ and } \rhk{|D|-b-1}{\Gamma_D}
  \neq 0.$$

  By \cref{e:Hochster}, this means that $$\beta_{a,|C|}(S/I) \neq 0
  \mbox{ and } \beta_{b,|D|}(S/I) \neq 0,$$ so that $t_a \geq |C|$ and
  $t_b \geq |D|$. Putting this all together we get
  $$t_a + t_b \geq |C| + |D| \geq n = t_i,$$
which settles our claim.
     \end{proof}

  \begin{discussion} Given a square-free monomial ideal $I$ if we are
    looking for top degree Betti numbers, by Hochster's formula
    (\cref{e:Hochster})
       $$\beta_{n-i-1,n}(S/I)=\dim_k\rhk{i}{\Gamma}.$$

       Now if $d$ is the smallest possible degree of a generator of
       $I$, then all monomials of degree $\leq d-1$ are not in $I$,
       which means all possible faces of dimension $ \leq d-2$ are in
       $\Gamma=\N(I)$. This means that the smallest index $i$ with $\rhk{i}{\Gamma}\neq 0$ is $d-2$, that is $$\rhk{i}{\Gamma}=0 \mbox{ for } i <d-2$$ and hence
$$\beta_{j,n}(S/I)=0 \mbox{ for } j=n-i-1>n-d+1.$$ So $n-d+1$ is the
       maximum homological degree where we could have a nonvanishing
       top degree Betti number.  We
       do not have an example of our setting where $n-d+1$ is not the
       projective dimension. After comparing with bounds on the projective dimension of $S/I$ given by  Dao and Schweig~\cite[Theorem~3.2,
         Remark~3.4]{DS} 
       in terms of dominance parameters of clutters, we concluded 
       that $n-d+1$ is often either the projective dimension of $S/I$
       or very close to it, though we were not able to determine how close. 
\end{discussion}

 \begin{example} Let $I=(xyz,xzv,xuv,yzu,yuv)$ be an ideal of
	$S=k[x,y,z,u,v]$ in $5$ variables. Here the smallest degree of a
	generator of $I$ is $d=3$, so $n-d+1=3$, so we pick $a=1$ and
	$b=2$. According to Macaulay2~\cite{M2} the Betti table of $S/I$
	is
	$$\begin{array}{rlllll}
	&0 &1 &2 &3\\
	\mbox{total}:&1 &5 &5 &1\\
	0:&1 &. &. &.\\
	1:&. &. &. &. \\
	2:&. &5 &5 &1
	\end{array}$$
	which verifies that
	$$t_3=5, t_2=4, t_1=3 \Longrightarrow t_3 < t_1+t_2 =7.$$ 
\end{example}

\begin{example}
	In \cref{e:run-example}, $I=(xz,yz,xu,yu,xv,yv,zuv)$ is a
        square-free monomial ideal in $5$ variables where $d=2$ and
        $n-d+1=4$. According to the Betti table of $I$, $t_4=5$,
        $t_3=5$, $t_2=4$ and $t_1=3$. Here $t_4<t_1+t_3=8$ and
        $t_4<2t_2=8$. Note that we also have $\beta_{3,5}(S/I)\neq 0$
        where $3 <4=n-d+1$ while still we have $t_3 < t_1+t_2=7$.
	\end{example}

\section{Special cases of breaking up simplicial homology}

  In this section, we consider breaking up special classes of cycles,
  where we can provide a combinatorial description for the
  lower-dimensional cycles.

\subsection{The case of a disconnected simplicial complex}

We begin with an example.

\begin{example}\label{e:ex-4.1}
	Let $\N(I)=\Gamma=\langle uv,xy,yz,xz \rangle$ be a simplicial complex on $n=5$ vertices.
	
	$$
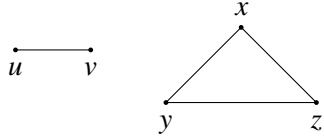
\begin{tikzpicture}

	\coordinate (u) at (-1,0.7);
	\coordinate (v) at (0,0.7);
	\coordinate (x) at (2,1);
	\coordinate (y) at (1,0);
	\coordinate (z) at (3,0);

	\draw (x)node[circle,thick,draw=black,fill=black,inner sep=0pt,label=$x$]{}--(y)node[circle,thick,draw=black,fill=black,inner sep=0pt,label=below:$y$]{}--(z)node[circle,thick,draw=black,fill=black,inner sep=0pt,label=below:$z$]{}--cycle;
	
	\draw (u)node[circle,thick,draw=black,fill=black,inner sep=0pt,label=below:$u$]{}--(v)node[circle,thick,draw=black,fill=black,inner sep=0pt,label=below:$v$]{};

\end{tikzpicture}  $$
	
Here $\rhk{0}{\Gamma}\neq 0$ and hence $\be_{4,uvxyz}(S/I)\neq 0$. If $4=a+b$, then using \cref{c:dual-version} we have
the following two cases to consider.

\begin{enumerate}
\item $a=1$ and $b=3$. Let $C=\{u,x\}$ and $D=\{u,v,y,z\}$. Then
  $C\cup D= \{u,v,x,y,z\}$, $C \cap D=\{u\}\in \Gamma$ and
$$\rhk{|C|-a-1}{\Gamma_C}=\rhk{0}{\langle u,x \rangle}\neq 0 \mbox{
  and } \rhk{|D|-b-1}{\Gamma_D}=\rhk{0}{\langle uv,yz \rangle}\neq
  0.$$

\item $a=b=2$. Let $C=\{u,x,v\}$ and $D=\{u,y,z\}$. Then $C\cup D=
  \{u,v,x,y,z\}$, $C \cap D=\{u\}\in \Gamma$ and 
$$\rhk{|C|-a-1}{\Gamma_C}=\rhk{0}{\langle uv,x \rangle}\neq 0 \mbox{
    and } \rhk{|D|-b-1}{\Gamma_D}=\rhk{0}{\langle u,yz \rangle}\neq
  0.$$
\end{enumerate}
\end{example}

In general if $\Gamma$ is a disconnected complex on $n$
  vertices with Stanley-Reisner ideal $I$, then $\beta_{n-1,n}(S/I)
  \neq 0$, and if $n-1=a+b$ for some $a,b>0$, then we can always find
  disconnected induced subcomplexes $\Gamma_C$ and $\Gamma_D$ where
  $C=a+1$ and $D=b+1$, as in the example above. Below we demonstrate
  how this can be done.

	If $\Gamma$ is  disconnected, then it has the form
		$$\Gamma=\Gamma_1 \cup \cdots \cup \Gamma_t$$	
	where $\Gamma_1,\ldots,\Gamma_t$ are connected components and
        $t > 1$. In this case, $|V(\Gamma_i)| \geq 1$ for all $1 \leq
        i \leq t$, $V(\Gamma)=V(\Gamma_1) \cup \cdots \cup
        V(\Gamma_t)$ and $V(\Gamma_k) \cap V(\Gamma_l) = \emptyset$
        for all $1 \leq k < l \leq t$.
	
	Without loss of generality and up to renaming the variables,
        we can assume the following:
	\begin{itemize}
	\item $|V(\Gamma_1)| \leq |V(\Gamma_2)| \leq \cdots \leq
          |V(\Gamma_t)|$,

        \item $x_k \in V(\Gamma_k)$ for $1 \leq k \leq t$,

        \item $V(\Gamma_1)=\{x_1,x_{t+1},\ldots,x_{t+|V(\Gamma_1)|-1}\}$

        \item $V(\Gamma_k)=\{x_k,x_{{\tiny (t+|V(\Gamma_1)|+\cdots+|V(\Gamma_{k-1})|-k+2)}},\ldots,x_{{\tiny (t+|V(\Gamma_1)|+\cdots+|V(\Gamma_{k})|-k)}} \}$ for each $1 < k \leq t$.
	\end{itemize}

   \begin{example}
	The simplicial complex $\Gamma$ in \cref{e:ex-4.1} can
        be relabeled and written as $\Gamma=\Gamma_1 \cup \Gamma_2$
        where $\Gamma_1=\langle x_1 x_3 \rangle$ and $\Gamma_2=\langle
        x_2x_4,x_4x_5,x_2x_5\rangle$.
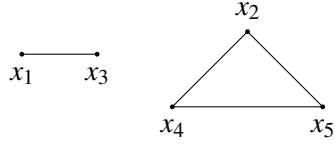
        $$\begin{tikzpicture}

        	\coordinate (x1) at (-1,0.7);
        	\coordinate (x3) at (0,0.7);
        	\coordinate (x2) at (2,1);
        	\coordinate (x4) at (1,0);
        	\coordinate (x5) at (3,0);

        	\draw (x2)node[circle,thick,draw=black,fill=black,inner sep=0pt,label=$x_2$]{}--(x4)node[circle,thick,draw=black,fill=black,inner sep=0pt,label=below:$x_4$]{}--(x5)node[circle,thick,draw=black,fill=black,inner sep=0pt,label=below:$x_5$]{}--cycle;
        	
        	\draw (x1)node[circle,thick,draw=black,fill=black,inner sep=0pt,label=below:$x_1$]{}--(x3)node[circle,thick,draw=black,fill=black,inner sep=0pt,label=below:$x_3$]{};

        \end{tikzpicture}  $$

	\end{example}

   For each $1 \leq a< n-1$, define $$C=\{x_1,x_2,
   \ldots,x_{a+1}\} \mbox{ and } D=\{x_1,x_{a+2},\ldots,x_n\}.$$
   Clearly $C\cup D=\{x_1,\ldots,x_n\}$, $|C|=a+1$, $|D|=n-a$ and
   $C\cap D=\{x_1\}\in \Gamma$. Moreover, it is easy to see that both
   $\Gamma_C$ and $\Gamma_D$ are disconnected induced subcomplexes of
   $\Gamma$ on the subsets $\{x_1,x_2,\ldots,x_{a+1}\}$ and
   $\{x_1,x_{a+2},\ldots,x_n\}$, respectively. Therefore, if $b=n-a-1$
   $$\rhk{|C|-a-1}{\Gamma_C}=\rhk{0}{\Gamma_C}\neq 0 \mbox{ and }
   \rhk{|D|-b-1}{\Gamma_D}=\rhk{0}{\Gamma_D}\neq 0.$$

\subsection{The case of a graph cycle}
  
Recall that a {\em cycle} in a graph $G$ is an ordered list of
distinct vertices $x_1,\ldots,x_n$ where the edges are $x_{i-1}x_i$ for $2\leq i \leq n$ and $x_nx_1$. Graph
cycles characterize nontrivial $1$-homology in simplicial complexes;
see for example Theorem~3.2 in~\cite{C}.
	
 Suppose $\Gamma$ is a simplicial complex on the set $\{x_1,\ldots,x_n\}$
 that is the support complex of a face-minimal graph cycle, so that $\rhk{1}{\Gamma}\neq 0$. This
   means that $\beta_{n-2,n}(S/I) \neq 0$. Suppose $n-2=a+b$ for some
   $a,b>0$.

 Without loss of
 generality, $\Gamma$ can be written in the form
$$\Gamma=\langle x_1x_2,x_2x_3, \ldots, x_{n-1}x_n,x_nx_1 \rangle.$$

For $1 \leq a < n-2$, define
	$$C=\{x_1,x_3,x_4,\ldots,x_{a+2}\} \mbox{ and } D=\{x_2,x_{a+3},\ldots,x_n\}.$$

Clearly, $C\cup D=\{x_1,\ldots,x_n\}$, $|C|=a+1$, $|D|=n-a-1$ and
$C\cap D=\emptyset \in \Gamma$. Moreover, it is easy to see that both
$\Gamma_C$ and $\Gamma_D$ are disconnected induced subcomplexes of
$\Gamma$ on the subsets $\{x_1,x_3,x_4,\ldots,x_{a+2}\}$ and
$\{x_2,x_{a+3},\ldots,x_n\}$, respectively. Therefore,
	$$\rhk{|C|-a-1}{\Gamma_C}=\rhk{0}{\Gamma_C}\neq 0$$
and 
	$$\rhk{|D|-b-1}{\Gamma_D}=\rhk{0}{\Gamma_D}\neq 0$$
where $b=n-a-2$.

\begin{example}
	Let $\N(I)=\Gamma=\langle x_1x_2,x_2x_3,x_3x_4,x_4x_5,x_1x_5 \rangle$. 
	
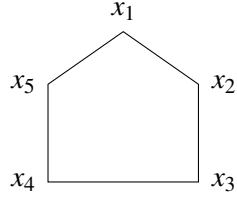
$$\begin{tikzpicture}

\coordinate (a) at (0,1);
\coordinate (b) at (1,0.3);
\coordinate (c) at (1,-1);
\coordinate (d) at (-1,-1);
\coordinate (e) at (-1,0.3);

\draw (a)node[inner sep=0pt,label=$x_1$]{}--(b)node[inner sep=0pt,label=right:$x_2$]{}--(c)node[inner sep=0pt,label=right:$x_3$]{}--(d)node[inner sep=0pt,label=left:$x_4$]{}--(e)node[inner sep=0pt,label=left:$x_5$]{}--cycle;

\end{tikzpicture}$$	
	
Then $\rhk{1}{\Gamma}\neq 0$ and hence $\be_{3,x_1\cdots x_5}(S/I)\neq 0$. Taking $a=1$ and $b=2$, set $C=\{x_1,x_3\}$ and $D=\{x_2,x_4,x_5\}$. Then 
	\begin{center}
	$\rhk{|C|-a-1}{\Gamma_C}=\rhk{0}{\langle x_1,x_3 \rangle}\neq 0$ and $\rhk{|D|-b-1}{\Gamma_D}=\rhk{0}{\langle x_2,x_4x_5 \rangle}\neq 0$.
	\end{center}

\end{example}

 \section{The case of general monomial ideals}\label{s:polarization}
 
The polarization~\cite{Fr} of a monomial ideal $I$ is a method to
transform $I$ to a square-free monomial ideal, by adding new variables
to the polynomial ring. The procedure is described below.

\begin{definition}[{\bf Polarization}]\label{d:polarization} Let $I$ be
  minimally generated by
  monomials $\m_1,\ldots\m_q$ in the polynomial ring $R=k[\xs]$. For
  $i\in \{1 ,\ldots,n\}$, let 
  
$$p_i =  \begin{cases}
  	1 & \mbox{ if } x_i \not \mid \m_u \mbox{ for every } u \in [q]\\
  	\max\left \{j \st x_i^j \mid \m_u
  	 \mbox{ for some } u \in [q] \right \} & \mbox{ otherwise.}
  	\end{cases}$$ 
  Let $S$ be the polynomial ring in $\mathbf{p}=p_1 +\cdots +p_n$
  variables $$S=k[x_{i,j} \st 1 \leq i \leq n, \ 1 \leq j \leq p_i]$$
  and let the {\bf polarization of $I$} be the square-free monomial
  ideal $$\PP(I)= \left( \PP(\m_1),\ldots,\PP(\m_q) \right ) $$ where,
  if $\m={x_{a_1}}^{b_1}\cdots {x_{a_c}}^{b_c}$ where the $a_i$ are
  distinct integers in $\{1,\ldots,n \}$ and $1 \leq b_i \leq p_i$ for
  $1\leq i \leq c$, then
  $$\PP(\m)=x_{a_1,1}\cdots x_{a_1,b_1} x_{a_2,1}\cdots x_{a_2,b_2} \cdots x_{a_c,1}\cdots x_{a_c,b_c}.$$
\end{definition}

\begin{example} If $I=(x^2,xy^3z^2) \subseteq k[x,y,z]$ then its polarization
  is the square-free monomial ideal $\PP(I)=(x_1x_2,x_1y_1y_2y_3z_1z_2)$ in the polynomial ring  $k[x_1,x_2,y_1,y_2,y_3,z_1,z_2]$.
\end{example}

\begin{corollary}[{\bf Subadditivity of syzygies of  monomial ideals}]\label{t:subadditivity-general}
  If $I$ is a monomial ideal in the polynomial ring $R=k[\xs]$ where
  $k$ is a field, $d$ is the smallest possible degree of a generator
  of $I$, and $\mathbf{p}$ is  defined  as in
  \cref{d:polarization}. Suppose $i=\mathbf{p}- d + 1$, $\beta_{i,\mathbf{p}}(R/I) \neq
  0$ and $i=a+b$, for some positive integers $a$ and $b$. Then $t_i
  \leq t_a + t_b$.
        \end{corollary}

\begin{proof} Let $I=(\m_1,\ldots,\m_q)$, whose polarization is the square-free
  monomial ideal $\PP(I)$ in the polynomial ring $S$ in $\mathbf{p}$
  variables in \cref{d:polarization}.  Since
  $\beta_{i,\mathbf{p}}(R/I)\neq 0$, we must have
  $\beta_{i,\m}(R/I)\neq 0$ for some $\m \in \LCM(I)$. On the other
  hand, $\mathbf{p}$ is the largest possible degree for a monomial in
  $\LCM(I)$, and so $\m = \lcm(\m_1,\ldots,\m_q)$, the top monomial in
  the lcm lattice of $I$.

  Now the two lcm
  lattices $\LCM(I)$ and $\LCM(\PP(I))$ are isomorphic (\cite{GPW}),
  and the degree $\mathbf{p}$ square-free monomial $\PP(\m)$ sits on
  top of the lattice $\LCM(\PP(I))$, and so
  $\beta_{i,\mathbf{p}}(S/\PP(I)) \neq 0$. Now since
  $\deg(\m_i)=\deg(\PP(\m_i))$ for all $1 \leq i \leq q$, the
  conditions for \cref{t:subadditivity} hold, and therefore $t_i \leq
  t_a + t_b$ holds for the ideal $\PP(I)$. But as the graded Betti
  numbers of $I$ and $\PP(I)$ are equal, the inequality also holds for
  $I$, and we are done.
\end{proof}

\begin{example} Let $I=(xy^2,xyz,y^3,y^2z)$ be an ideal of $R=k[x,y,z]$. Here $\mathbf{p}=5$ and the smallest degree of a generator is $d=3$, so $\mathbf{p}-d+1=3$. We pick $a=1$ and $b=2$. According to Macaulay2~\cite{M2} the Betti table of $R/I$
		is
		$$\begin{array}{rlllll}
			&0 &1 &2 &3\\
			\mbox{total}:&1 &4 &4 &1\\
			0:&1 &. &. &.\\
			1:&. &. &. &. \\
			2:&. &4 &4 &1
		\end{array}$$
		which verifies that
		$$t_3=5, t_2=4, t_1=3 \Longrightarrow t_3 < t_1+t_2 =7.$$ 
 \end{example}

\section{Final Remarks}

\cref{q:main}, \cref{q:main-general},
\cref{q:main-H-link} and \cref{q:main-H} are all equivalent, though
their different settings allow the application of  different
(inductive) tools. All of them are open in their full generality as
far as we know, though each can be answered positively for certain
classes of ideals or combinatorial objects. A positive answer to
either would settle the subadditivity question for monomial ideals in
a polynomial ring.


\end{document}